\documentclass[11pt]{amsart}
\usepackage{txfonts}
\usepackage{amscd}
\usepackage{cite}
\usepackage{epsfig}
\usepackage{verbatim}
\usepackage{mathdots}
\usepackage{amssymb}
\usepackage{amsfonts}
\usepackage{amsbsy}
 \usepackage{graphicx}
 \usepackage[usenames]{color} 
 \usepackage{cancel}

\newtheorem{theo}{Theorem}[section]
\newtheorem{lem}[theo]{Lemma}

\newtheorem{coro}[theo]{Corollary}

\theoremstyle{definition}
\newtheorem{defi}[theo]{Definition}

\numberwithin{equation}{section}

\newcommand{\R}{{\mathbb R}}
\newcommand{\Z}{{\mathbb Z}}

\newcommand{\Q}{{\mathbb Q}}

\begin{document}

\title{Spectra of Cantor measures}

\author{Xin-Rong Dai}
\address{School of Mathematics and Computational Science, Sun Yat-sen University, Guangzhou, 510275,  P. R. China}
\email{daixr@mail.sysu.edu.cn}
\thanks{The research is partially supported by the NSF of
China (No. 11371383) and the Fundamental Research Funds for the Central Universities.}
\subjclass[2010]{primary 42A65, 42B05, secondary 28A78, 28A80.}
\keywords{Spectral measure, Cantor measure, Bernoulli convolution, spectrum, maximal orthogonal set.}

\maketitle

\begin{abstract}
A spectrum of a probability measure $\mu$ is a countable set $\Lambda$ such that  $\{\exp(-2\pi i\lambda \cdot), \ \lambda \in \Lambda \}$ is an orthogonal basis for $L^2(\mu)$.
In this paper, we consider the problem when a countable set become the spectrum of
the Cantor measure.
Starting from tree labeling  of a maximal orthogonal set, 
 we introduce a new quantity to measure minimal level difference
 between a branch of the labeling tree and its subbranches.
Then we use boundedness and linear increment of that level difference measurement
to justify whether a given maximal orthogonal set is a spectrum or not.
This together with the tree labeling of a maximal orthogonal set provides fine structures of spectra of  Cantor measures.
 As applications of our justification, we provide a characterization for the integrally expanding set $K\Lambda$ of a spectrum $\Lambda$
  to be  a spectrum again, thereby we find all integers $K$ such that
  $K\Lambda_4$ are spectra of the $1/4$-Cantor measure $\mu_{4}$, where
  $\Lambda_4:=\{\sum_{n=0}^\infty d_n 4^n: d_n\in \{0, 1\}\}$
  is the first known spectrum for $\mu_{4}$.
Furthermore, we construct  a spectrum $\Lambda$  such that
the integrally shrinking set
$\Lambda/K$ is a maximal orthogonal set but not a spectrum for some integer $K$.
  \end{abstract}


\section{Introduction}
\label{introduction.section}
\setcounter{equation}{0}

A fundamental problem in harmonic analysis is whether $\{\exp({-2\pi i\lambda x}), \ \lambda\in\Lambda\}$ is an orthogonal basis of
 $L^2(\mu)$, the space of all square-integrable functions with respect to
a probability measure $\mu$.
The above probability measure $\mu$ is known as a {\em spectral measure} and the countable set $\Lambda$ as its {\em spectrum}.
Spectral theory for the Lebesgue measures on sets has been studied extensively since
it initialed by Fuglede 1974 \cite{F}, see
 \cite{LW3,  Fal, T3} and references therein. Recently, He, Lai and Lau \cite{HLL} proved that a spectral measure is pure type (i.e. either absolutely continuous or singular continuous or counting measure).
%
For singular continuous measures, the first spectral measure was found
by Jorgenson and Pederson in 1998 \cite{JP}, they proved that $\Lambda_4:=\{\sum_{n=0}^\infty d_n 4^n: d_n\in \{0, 1\}\}$ is a spectrum of the Bernoulli convolution $\mu_4$.
Since then, 
some significant progresses have been made  and
 various new phenomena  different from spectral theory for the Lebesgue measure
 have been discovered \cite{JP, DuJ09, HuL, DS, D, DHL, DHL2, DHS, DHS13, DHLau, DL,  LW, S}. For instance,
Fourier frames on  the unit interval $[0, 1)$  have  Beurling dimension one \cite{Lan},
  while spectra of a singular measure could have zero  Beurling dimension \cite{DHL}. Here we define the \emph{Cantor measure} $\mu_{q,b}$ with $2\le q\in \Z$ and  $q<b\in \R$,
\begin{equation} \label {eq1.1}
\mu_{q,b}=\frac{1}{q}\sum_{i=0}^{q-1}\mu_{q,b}\big( f_i^{-1}(\cdot)\big),
\end{equation}
is a self-similar probability measure
associated with the iterated function system,
$$
f_i(x)=x/b+i/q, \  \ i=0,1,\dots, q-1.
$$
And we call the special case $\mu_b:=\mu_{2,b}$ the {\em Bernoulli convolutions}.
In 1998, Jorgenson and Pederson proved in their seminal paper \cite{JP} that
Bernoulli convolutions $\mu_b$ with $b\in 2\Z$  are  spectral measures. 
 The converse problem stood for a long time and it was solved in \cite{D} by
 the author in 2012 after important contributions by Hu and Lau \cite{HuL}.
 The complete characterization
 for  Bernoulli convolutions  in \cite{D}
  was recently extended  by He, Lau and the author \cite{DHLau} to  the Cantor measure $\mu_{q, b}$
    that  it is a spectral measure if and only if 
 \begin{equation}\label{bq.assumption}
  2\le q,\ b/q\in  \Z.\end{equation}
The Cantor measures $\mu_{q, b}$ with $q$ and $b$ satisfying \eqref{bq.assumption} are few of known singular spectral measures, but the
structure of their spectra is little known, even for the Bernoulli convolution $\mu_4$.
In this paper, we explore fine structure of spectra of these Cantor measures.

\smallskip

  Our exploration starts from  tree structure of a {\em  (maximal) orthogonal set} $\Lambda$, 
   meaning that
   $\{\exp({-2\pi i\lambda x}), \ \lambda\in\Lambda\}$ is a (maximal) orthogonal set of
 $L^2(\mu_{q, b})$. In 2009, Dutkay, Han and Sun gave a complete characterization of the
maximal orthogonal sets of the Bernoulli convolution $\mu_4$ by
 introducing  a  tree labeling tool \cite{DHS}.
Recently,  He, Lai and the author developed a tree labeling technique for
 Cantor measures $\mu_{q,b}$ \cite{DHL}. They proved that
a countable set is a maximal orthogonal set of the Cantor measure $\mu_{q, b}$
 if and only if it can be labeled as a maximal tree, see Theorem \ref{th1.6}.
Thus maximal orthogonal sets have tree structure and they can be built selecting maximal tree appropriately. While a maximal orthogonal set
 is not necessarily a spectrum since it may lack of completeness in $L^2(\mu_{q, b})$. The completeness of maximal orthogonal sets for  Cantor measures $\mu_{q, b}$
 is quite challenging, see
\cite{DHL, DHLau, DHS, DHS13,   JP, DuJ09, S} for various sufficient and necessary conditions. In fact,
 the completeness of exponential sets
is a classical problem in 
 Fourier analysis since 1930s', see  \cite{PW, Lev, OS,Po1, Po2,Y} and references therein for historical remarks and recent advances.

\medskip

The main contribution of this paper is to introduce a quantity ${\mathcal D}_{\tau, \delta}$
to measure minimal level difference between a branch $\delta$ of the labeling tree and its  subbranches, see
Definition \ref{taudistance.definition}.
We show in Theorem \ref{main.thm} that a maximal orthogonal set $\Lambda$
with maximal tree labeling $\tau$ is a spectrum if
 ${\mathcal D}_{\tau, \delta}$ is uniform bounded on all tree branches $\delta$,
and also in Theorem \ref{main2.thm.old} that
it is not a spectrum if 
${\mathcal D}_{\tau, \delta}$
increases linearly to the level
of the tree branches $\delta$.

Unlike spectra of the Lebesgue measure on the unit interval,
a spectrum $\Lambda$ of a singular measure could have the integrally rescaled set $K\Lambda$ being  its spectrum too, see
\cite{JP, DuJ09, DHS13} for  the Bernoulli convolution $\mu_4$.
In this paper, we apply our completeness results  in Theorems \ref{main.thm} and \ref{main2.thm.old}
to characterize
the spectral property of the  rescaled set $K\Lambda$ for a given spectrum $\Lambda$ of the Cantor measure $\mu_{q, b}$
via no repetend of $K$ for the labeling tree of $\Lambda$. As corollaries, we find all integers $K$ such that
  $K\Lambda_4$ are spectra of the Bernoulli convolution $\mu_{4}$, see Corollary \ref{cantor24.cor},
and we 
  construct a spectrum $\Lambda$ of the Cantor measure $\mu_{q, b}$  such that
the rescaled set $\Lambda/(b-1)$
 is its maximal orthogonal set but not its spectrum, see Theorems \ref{multiple.thm}, \ref{App.thm} and \ref{counterexample}.

\medskip

 This paper is organized as follows.
 In Section \ref{main.section}, we recall some preliminaries about (maximal) orthogonal
 sets for Cantor measures, and state our main results.
 In Sections \ref{sufficient.section}, we consider the problem when
 a maximal orthogonal set is a spectrum.
 In Section \ref{necesarry.section}, we discuss the necessity for
 a maximal orthogonal set to be a spectrum. 
In Section \ref{rescaledspectra.section}, we discuss rationally  rescaling  of a spectrum.

\section{Preliminaries and main theorems}
\label{main.section}

%
%

 We start this section from recalling a characterization of orthogonal sets of a probability measure $\mu$
 via its Fourier transform $\hat \mu$,
$$\hat{\mu}(\xi):=\int_{\R} e^{-2\pi i \xi x}d\mu(x).$$
Observe that the zero set of the Fourier transform $\widehat \mu_{q, b}$, see  \eqref{eq2.c} for its explicit expression,
 is
$$Z_{q, b}=\{b^ja:\  a\in \Z\backslash q\Z, 0\le j\in \Z\}.$$
%
%
%
Then  a discrete set $\Lambda$ is an orthogonal set of $ \mu_{q, b}$ if and only if we have the following for orthogonal sets of the Cantor measure $\mu_{q, b}$ \cite{DHL, DHLau}:  
 \begin{equation} \label {orthogonalcharacterization}
\Lambda-\Lambda \subset Z_{q, b} \cup \{0\}.
\end{equation}
As orthogonal sets  (maximal orthogonal sets and  spectra) are invariant under translations, in this paper we always
 normalize them by assuming that
\begin{equation} \label{integrallambda}0\in\Lambda\subset \Z.\end{equation}

\medskip

To introduce the tree structure of the maximal orthogonal set  of the Cantor measure $\mu_{q, b}$, we need some notation and concepts.
Denote $\Sigma_q:=\{0,\cdots, q-1\}$, and $\Sigma_q^n:=\underbrace{\Sigma_q\times\cdots\times \Sigma_q}_n, 1\le n\le \infty$
be the $n$ copies of $\Sigma_q$, and $\Sigma_q^{\ast} := \bigcup_{1\le n<\infty}\Sigma_q^n$.
Given $\delta=\delta_1\delta_2\cdots\in \Sigma_q^{\ast} \cup \Sigma_q^{\infty}$ and $\delta'\in\Sigma_q^{\ast}$, we define $\delta'\delta$ is the concatenation of $\delta'$ and $\delta$, and adopt the notation $0^\infty=000\cdots$, $0^k=\underbrace{0\cdots0}_k$. We call an element in $\Sigma_q^{\ast}\cup\Sigma_q^{\infty}$  as a {\em tree branch}.
 For each tree branch $\delta=\delta_1\delta_2\cdots$, denote 
\begin{equation}
\nonumber
{\delta|_k}:=
\left\{\begin{array}{ll}
\delta_1\delta_2\cdots\delta_k & {\rm when } \ \delta\in\Sigma_q^{\infty}, {\rm and } \\
(\delta0^\infty)|_k & {\rm when}\    \delta\in\Sigma_q^{\ast},
\end{array}\right. \end{equation}
for all $k\ge1$.

\begin{defi}\label{treemapping.def}
For $2\le q, b/q\in \Z$,  we say that
a mapping $\tau:\Sigma_q^\ast\to \{-1, 0, \ldots,
b-2\}$ is  a {\em tree mapping}
if
\begin{itemize}
\item[{(i)}] $\tau(0^n)=0$  for all $n\ge 1$,   and
\item[{(ii)}]
$\tau (\delta) \in \delta_n+q\Z$
if $\delta=\delta_1\cdots\delta_n\in \Sigma_q^n, n\ge 1$,
\end{itemize}
and that a tree mapping $\tau$ is   {\em maximal}
if \begin{itemize}
\item[{(iii)}] for any $\delta\in \Sigma_q^{\ast}$  there exists
$\delta'\in \Sigma_q^\ast$ such that
  $\tau((\delta\delta')|_n)=0$ for sufficiently large integers $n$.
  \end{itemize}
\end{defi}

In \cite{DHL}, He, Lai and the author  established the following
characterization for
 a maximal orthogonal set  of the Cantor measure $\mu_{q, b}$
 via some maximal tree mapping.

\begin{theo}\label{th1.6}{\rm (\!\cite{DHL})}\
Let $2\le q, b/q\in \Z$. Assume that $\Lambda$  is a countable set of real numbers containing zero.
 Then
$\Lambda$ is a maximal orthogonal set of the Cantor measure $\mu_{q, b}$ if and only if there exists a maximal tree mapping $\tau$ such that
$\Lambda = \Lambda(\tau)$, where
\begin{equation}\label{lambdatau.def}
\Lambda(\tau):=\Big\{\sum_{n=1}^\infty \tau(\delta|_n)b^{n-1}:\
  \delta \in \Sigma_q^\ast \ \ \text{such that} \ \ \tau(\delta|_m)=0 \ \ \text{for sufficiently large} \ m\Big\}.
\end{equation}\end{theo}

Given a maximal tree mapping $\tau: \Sigma_q^\ast\to \{-1, 0, \ldots, b-2\}$, we say that
$\delta\in \Sigma_q^n, n\ge 1$, is a {\em $\tau$-regular branch}
if
$\tau (\delta|_m)=0$ for sufficiently large $m$.
Define $\Pi_{\tau,  n}:  \Sigma_q^{\ast}\cup \Sigma_q^{\infty}\to \R, n\ge 1$,  by
\begin{eqnarray}\label{formally,n}
\Pi_{\tau, n}(\delta)=\sum_{k=1}^n \tau( \delta|_k)b^{k-1}.
\end{eqnarray}
One may verify that the restriction of $\Pi_{\tau, n}$ onto $\Sigma_q^n$ is one-to-one for any $n\ge 1$.
For a $\tau$-regular tree branch $\delta\in
\Sigma_q^\ast$, we can extend the definition $\Pi_{\tau, n}(\delta), n\ge 1$,
in \eqref{formally,n}
to $n=\infty$ by taking limit in \eqref{formally,n},
\begin{eqnarray}\label{formally}
\Pi_{\tau, \infty}(\delta) :=\sum_{k=1}^\infty \tau(\delta|_k)b^{k-1}.
\end{eqnarray}
  Applying the above $b$-nary expression, we conclude that a maximal orthogonal set of the Cantor measure $\mu_{q, b}$ is
the image of $\Pi_{\tau, \infty}$ for some maximal tree mapping $\tau$, 
\begin{equation*}
\Lambda(\tau)  = \big \{\Pi_{\tau, \infty} (\delta): \ \delta\in \Sigma_q^\ast \ \text{\rm are $\tau$-regular \ branches}
\big\}.
\end{equation*}
This together with Theorem \ref{th1.6} suggests that
various maximal orthogonal sets of the Cantor measure $\mu_{q, b}$  could be constructed by selecting maximal tree mapping appropriately.

\medskip


Now we introduce a  quantity to  measure (minimal) level difference between a tree branch and its subbranches,
which plays important role in our study of spectral property of Cantor measures.
For $\delta'\in \Sigma_q^\ast$ and  $\delta\in \Sigma_q^n$ for some $n\ge 1$, define
\begin{equation}\label{Ddelta.def0}
\mathcal{D}_{\tau, \delta} (\delta')=
 \# A_\delta(\delta') +\sum_{n_j\in B_\delta(\delta')}({n_j}-{n_{j-1}}-1),
\end{equation}
where
$A_\delta(\delta'):=\{m\ge 1: \ \tau(\delta \delta'|_m)\ne 0\}$,
$B_\delta(\delta'):=\{m\ge 1: \  \tau(\delta \delta'|_m)\not\in q\Z\}$,  $n_0=0$, and
 $\{n_j\}_{j\ge 1}$ is a strictly increasing sequence of positive integers
given by $\{n_j: j\ge 1\}=A_\delta(\delta')$, and $\# E$ is the cardinality of a set $E$.

\begin{defi}\label{taudistance.definition}
Let $2\le q, b/q\in \Z$ and $\tau: \Sigma_q^\ast\to \{-1, 0, \ldots, b-2\}$ be a maximal tree mapping.
Define 
\begin{equation}\label{Ddelta.delta}
\mathcal{D}_{\tau, \delta} := \inf \big\{{\mathcal D}_{\tau, \delta} (\delta'): \
\delta'\in \Sigma_q^\ast \big\}, \ \ \delta\in \Sigma_q^\ast.
\end{equation}
\end{defi}

Given a maximal tree mapping $\tau: \Sigma_q^\ast\to \{-1, 0, \ldots, b-2\}$, we say that
$\delta\in \Sigma_q^n, n\ge 1$, is a  a {\em $\tau$-main branch}
if
$\tau (\delta|_m)=0$  for all $m>n$.
Clearly $\delta\in \Sigma_q^\ast$ is a $\tau$-regular branch if and only if either $\delta$
 is a $\tau$-main branch or  $\delta0^k$  is for some $k\ge 1$; and  for any $\delta\in \Sigma_q^\ast$ there exists a
  $\tau$-main subbranch $\delta \delta'$, where $\delta'\in \Sigma_q^*$.
For any $\delta\in \Sigma_q^\ast$,  one may verify that the quantity ${\mathcal D}_{\tau, \delta}$ 
is the minimal   distance to its $\tau$-main subbranches,
\begin{equation} \label{Ddelta.delta2}
\mathcal{D}_{\tau, \delta} =  \inf \big\{{\mathcal D}_{\tau, \delta} (\delta'): \ \delta\delta'\
\text{are $\tau$-main branches}\big\}<\infty.
\end{equation}

\medskip

A challenging problem in spectral theory for
the Cantor measure $\mu_{q, b}$
 is when a maximal orthogonal set becomes a spectrum 
\cite{DHL, DHLau, DHS, DHS13,   JP, DuJ09, S}.
Now we present our main results of this paper. In our first main result, a sufficient condition via boundedness of $\mathcal{D}_{\tau, \delta}, \delta\in \Sigma_q^\ast$,
is provided  for   a maximal orthogonal set of the Cantor measure $\mu_{q, b}$
to its spectrum.

\begin{theo}\label{main.thm}
Let $2\le q, b/q\in \Z$. If  $\tau: \Sigma_q^\ast\to \{-1, 0, \ldots, b-2\}$ is a maximal tree mapping
such that
\begin{equation}\label{main.thm.eq1}
  {\mathcal D}_\tau:=\sup \{\mathcal{D}_{\tau, \delta}: \delta \in \Sigma_q^{\ast} \}<\infty,
  \end{equation}
then $\Lambda(\tau)$ in \eqref{lambdatau.def} is a spectrum of the Cantor measure $\mu_{q, b}$.
\end{theo}

We believe that the boundedness assumption on ${\mathcal D}_{\tau, \delta}, \delta\in \Sigma_q^\ast$,
 is a very weak sufficient condition for a maximal orthogonal set to be a spectrum.
In fact, as  shown in the next theorem, the above boundedness condition on  ${\mathcal D}_{\tau, \delta}, \delta\in \Sigma_q^\ast$,
 is close to be necessary.


\begin{theo}\label{main2.thm.old}
Let $2\le q, b/q\in \Z$,   $\tau: \Sigma_q^\ast\to \{-1, 0, \ldots, b-2\}$ be a maximal tree mapping.
 If there exists a positive number $\epsilon_0$ such that for each $n\ge1$ and $\delta=\delta_1\delta_2\cdots\delta_n\in \Sigma_q^n$ with $\delta_n\ne0$,
\begin{equation}\label{main2.thm.old.eq1}
{\mathcal D}_{\tau, \delta}\ge \epsilon_0 n,
\end{equation}
 then
$\Lambda(\tau)$ in \eqref{lambdatau.def} is not a  spectrum of the Cantor measure  $\mu_{q, b}$.
\end{theo}

Finally we apply our completeness results  in Theorems \ref{main.thm} and \ref{main2.thm.old} to
 the rescaling-invariant problem when the
 rescaled set $K\Lambda$ is a spectrum of the Cantor measure $\mu_{q, b}$ if $\Lambda$ is.
This simple and natural way  to construct new spectra
  from  known ones
  is motivated from the conclusion that if $K=5^k$ for some $k\ge 1$,
  then
 the rescaled set
  $K\Lambda_4: =\{K\lambda:\ \lambda\in \Lambda_4\}$
 of  the  spectrum 
\begin{equation}\label{lambda4.def}
\Lambda_4:=\Big\{\sum_{j=0}^\infty d_j 4^j, d_j\in \{0, 1\}\Big\}\end{equation}
of the Bernoulli convolution $\mu_4$
 is also a spectrum
 \cite{JP, DuJ09, DHS13}.
In the next theorem, we show that if the maximal tree mapping $\tau$
associated with the spectrum $\Lambda$ satisfies the boundedness assumption \eqref{main.thm.eq1}, then
the integrally rescaled  set $K\Lambda$ is a spectrum of the Cantor measure $\mu_{q, b}$ if and only if  it is a maximal orthogonal set.


\begin{theo}\label{multiple.thm}
Let $2\le q, b/q\in \Z$,  $\tau: \Sigma_q^\ast\to \{-1, 0, \ldots, b-2\}$ be a maximal tree mapping
satisfying \eqref{main.thm.eq1}, and $\Lambda(\tau)$  be as in \eqref{lambdatau.def}.
Then  for any  integer $K$ being prime with $b$,
 $K\Lambda(\tau)$ is a spectrum of the Cantor measure $\mu_{q,b}$
 if and only if it is a maximal orthogonal set.
\end{theo}



%

Applying Theorem \ref{multiple.thm}, we find all possible integers $K$ such that $K\Lambda_4$ are spectra of
 the  Bernoulli convolution $\mu_4$, c.f.  \cite{JP, DuJ09, DHS13}.

 \begin{coro}\label{cantor24.cor}
 Let $\Lambda_4$ be as in \eqref{lambda4.def} and $K\ge 3$ be an odd integer. Then $K\Lambda_4$ is a spectrum of the Bernoulli convolution $\mu_4$ if and only if
 there does not  exist a positive integer $N$ such that
 \begin{equation}
 K \sum_{j=1}^{N} d_j 4^{j-1}\in (4^N-1)\Z\backslash\{0\}
 \end{equation}
 for some $d_j\in \{0, 1\}, 1\le j\le N$.
 \end{coro}

Given a spectral set $\Lambda$ of the Cantor measure $\mu_{q, b}$,
 its irrational
rescaling set $r\Lambda$ (i.e., $r\not\in \Q$) is not an orthogonal set (and hence not a spectrum) by \eqref{integrallambda}.
The next question is when a rational rescaling set $r\Lambda$
is an orthogonal set, or a maximal orthogonal set, or a spectrum. A necessary condition is that
$r\Lambda\subset \Z$ by \eqref{integrallambda},
  but unlike integral rescaling discussed in Theorems \ref{multiple.thm} there are lots of interesting problems unsolved yet.
%
%
%
In this paper, we apply Theorems \ref{main.thm} and \ref{main2.thm.old} to
  construct a spectrum $\Lambda$ of the Cantor measure $\mu_{q, b}$  such that
the rescaled set $\Lambda/(b-1)$
 is its maximal orthogonal set but not its spectrum, see Theorem \ref{counterexample}.

\section{Maximal orthogonal sets and spectra: a sufficient condition}
\label{sufficient.section}

In this section, 
 we prove  Theorem \ref{main.thm}.
For that purpose, we need several technical lemmas
on spectra of the Cantor measure $\mu_{q, b}$, a crucial lower bound estimate for its Fourier transform
$\widehat \mu_{q, b}$,
and an identity for multi-channel conjugate quadrature filters.

%

\medskip

For 
 an orthogonal set $\Lambda$ of $L^2(\mu_{q, b})$ containing zero, let
 \begin{equation}\label{Q.def}
 Q_{\Lambda}(\xi):=\sum_{\lambda \in \Lambda} |\widehat{\mu_{q, b}}(\xi+\lambda)|^2.
 \end{equation}
Then $Q_{\Lambda}$ is a real analytic function on the real line with $Q_\Lambda(0)=1$, and
 $$Q_{\Lambda}(\xi)= \sum_{\lambda\in \Lambda}
  |\langle e_\lambda, e_{-\xi}\rangle|^2\le \|e_{-\lambda}\|^2= 1, \ \xi\in \R,$$
where the equality holds if $\Lambda$ is a spectrum.
 The converse is shown to be true in \cite{JP, DHL}. This provides a characterization for an orthogonal set
 of the Cantor measure $\mu_{q, b}$ to be its spectrum.

\begin{lem}\label{jp.lem} {\rm (\!\cite{JP, DHL})}\
Let  $2\le q,  b/q \in \Z$, 
and let  $Q_\Lambda(\xi)$  be defined by \eqref{Q.def}. 
Then  an orthogonal set $\Lambda$ of the Cantor measure $\mu_{q, b}$
 is a spectrum  if and only if
$Q_{\Lambda}(\xi)=1$ for all $\xi \in \R$.
\end{lem}

\medskip
For the Cantor measure $\mu_{q, b}$,
taking Fourier transform at both sides of the equation \eqref{eq1.1} leads to the following refinement equation in the Fourier domain:
 \begin{equation}\label{eq2.0}
\widehat{\mu_{q,b}}(\xi)=H_{q, b}(\xi/b)\cdot \widehat{\mu_{q,b}}(\xi/b),
\end{equation}
where
\begin{equation}
\label{hdb.def} H_{q, b}(\xi):=\frac{1}{q}\sum_{l=0}^{q-1}e^{-2\pi i lb \xi /q} 
\end{equation}
is  a  periodic function 
 with the properties that $H_{q, b}(0)=1$,
  \begin{equation}\label{hqb.property2} H_{q, b}(\xi)=0 \ {\rm if \ and \ only \ if} \ b\xi\in \Z\backslash q\Z,\end{equation}
  and
\begin{equation}\label{hqb.property3}
H_{q, b}^\prime(j/b)\ne 0\ {\rm  for\ all} \  j\in \Z.
\end{equation}
Applying \eqref{eq2.0} repeatedly 
  and then taking limit $m\to\infty$, we obtain an explicit expression
for the Fourier transform of the Cantor measure $\mu_{q,b}$:
\begin{equation}\label{eq2.c}
\widehat{\mu_{q,b}}(\xi) = H_m(\xi) \widehat{\mu_{q,b}}(\xi/b^m)= \prod_{j=1}^\infty H_{q,b}(\xi/b^j), \ \ m\ge 1,
\end{equation}
where
\begin{equation}\label {hm.def}
H_m(\xi):=\prod_{j=1}^{m} H_{q, b}(\xi/b^j),\ m\ge 1.
\end{equation}
Let  $2\le q,  b/q \in \Z$.  Define
\begin{equation}\label{r0r1.def} r_0:=\inf_{|\xi|\le (b-2)/(b-1)} |\widehat \mu_{q, b}(\xi)|
\ \ {\rm and}\  \ r_1:= \inf_{1\le j\le q-1} \inf_{|\xi|\le (b-2)/(b-1)} |\xi|^{-1} |H_{q, b}(\xi/b+j/b)|.
\end{equation}
Then it follows from \eqref{hqb.property2}, \eqref{hqb.property3} and
\eqref{eq2.c}  that
 both $r_0$ and $r_1$ are well-defined and positive,
\begin{equation}\label{r0r1.positive} r_0>0 \ \ {\rm and}\  \ r_1>0.
\end{equation}
Set
\begin{equation}\label{tb.def}
T_b= \Big(-\frac {1}{b-1}, \frac {b-2}{b-1} \Big ) \Big\backslash   \Big(-\frac {1}{b(b-1)}, \frac {b-2}{b(b-1)} \Big ).
\end{equation}
For any $m\ge 1$ and  $d_j\in \{-1, 0, \ldots, b-2\}, 1\le j\le m$, with $d_m\ne 0$,  one may verify that
\begin{equation}\label{tb.property} \Big(\xi+\sum_{j=1}^m d_j b^{j-1}\Big)b^{-m}\in T_b \ {\rm for \ all} \  \xi\in
\Big(-\frac {1}{b-1}, \frac {b-2}{b-1} \Big ).
\end{equation}
%

To prove Theorem \ref{main.thm}, we need the following two lemmas which are related to the lower bound estimates of
$|\widehat{\mu_{q,b}}(\xi+\lambda)|$ for  $\xi\in T_b$ and $\lambda\in \Z$.

\begin{lem}\label{lem3.5}
Let $2\le q,  b/q \in \Z$, $\mu_{q, b}$ be the Cantor measure in \eqref{eq1.1}, and
let $\lambda=\sum_{j=1}^K d_{n_j} b^{n_j-1}$ for some positive  integers $n_j, 1\le j\le K$, satisfying $0=:n_0< n_1<\ldots<n_K$, and
for some  $d_{n_j}, 1\le j\le K$, belonging to the set $\{-1, 1, 2, \ldots, b-2\}$. Then
\begin{equation} \label {lem3.5.eq1}
|\widehat{\mu_{q,b}}(\xi+\lambda)| \ge
  r_0^{K+1} \Big(\frac{r_1}{b(b-1)}\Big)^{\# B} b^{-\sum_{j\in  B} (n_j-n_{j-1}-1)},\ \xi\in T_b,
\end{equation}
where  $B=\{1\le j\le K: \ d_{n_j}\not\in   q\Z\}$
and $r_0, r_1$ are given in \eqref{r0r1.def}.
\end{lem}

\begin{proof}
For $0\le i\le K$, define $\xi_0=\xi$ and
$\xi_i=(\xi+\sum_{j=1}^{i} d_{n_j} b^{n_j-1})/b^{n_{i}}$ for $1\le i\le K$.
Then
\begin{equation}\label{xii.tb} \xi_i\in T_b\ \ \text{for all}\ \ 0\le i\le K
\end{equation}
by \eqref{tb.property}.
Observe that
\begin{equation}\label{Hqbproperty.one}
|H_{q, b}(\eta)|\le 1 \ \text{for  all} \ \eta\in \R\ \text{and}
\ \sup_{b\eta\in T_b} |H_{q, b}(\eta)|<1.
\end{equation}
The above observation, together with   \eqref{eq2.c}, \eqref{xii.tb}
 and the fact that
$H_{q,b}$ has period $q/b$, implies
\begin{eqnarray*}\nonumber
\prod_{\ell=n_{i-1}+1}^{n_{i}}  \big| H_{q, b}\big((\xi+\lambda)/b^{\ell}\big)\big|
& =   &
\prod_{\ell=n_{i-1}+1}^{n_{i}}\Big |H_{q, b}\Big(\Big (\xi+\sum_{j=1}^{i-1} d_{n_j} b^{n_j-1}+d_{n_i} b^{n_i-1} \Big)b^{-\ell}\Big)\Big|\\
& =   & \prod_{\ell'=1}^{n_{i}-n_{i-1}}
\big|H_{q, b}\big(\xi_{i-1}/b^{\ell'}\big)\big| \ge
 |\widehat{\mu_{q,b}}(\xi_{i-1})|
\ge r_0
\end{eqnarray*}
if  $d_{n_i}\in  q\Z$;
and
\begin{eqnarray*} & &
\prod_{\ell=n_{i-1}+1}^{n_{i}}  \big| H_{q, b}\big((\xi+\lambda)/b^{\ell}\big)\big|\nonumber\\
& =   &\Big(
\prod_{\ell=n_{i-1}+1}^{n_{i}-1}\Big |H_{q, b}\Big(\Big (\xi+\sum_{j=1}^{i-1} d_{n_j} b^{n_j-1}\Big)b^{-\ell}\Big)\Big|\Big)
\cdot \Big |H_{q, b}\Big(\Big (\xi+\sum_{j=1}^{i-1} d_{n_j} b^{n_j-1}+d_{n_i} b^{n_i-1}\Big)b^{-n_i}\Big)\Big|\nonumber
\\
& \ge &   |\widehat{\mu_{q,b}} \big (\xi_{i-1})| \cdot  \big | H_{q, b}(\xi_{i-1}/b^{n_i-n_{i-1}}+d_{n_i}/b)\big|
\ge   r_0 r_1 |\xi_{i-1}|/b^{n_i-n_{i-1}-1}\nonumber\\
& \ge &  r_0r_1 b^{-n_i+n_{i-1}}/(b-1)
\end{eqnarray*}
if $d_{n_i} \not\in q\Z$.
Combining the above two lower bound estimates with
\begin{equation}
\widehat{\mu_{q,b}}(\xi+\lambda) = 
\Big( \prod_{i=1}^K
 \prod_{\ell=n_{i-1}+1}^{n_i} H_{q, b} \big((\xi+\lambda)/b^\ell\big) \Big)
\cdot \widehat {\mu_{q, b}}\big( (\xi+\lambda)/b^{n_K}\big)
\end{equation}
proves  \eqref{lem3.5.eq1}.
\end{proof}

\begin{lem}\label{Add}
Let $2\le q, b/q\in \Z$, and
  $\tau:\Sigma_q^\ast\to\R$ be a maximal tree mapping satisfying \eqref{main.thm.eq1}. Then for each $\delta\in\Sigma_q^M$, $M>0$, there exists $\delta'\in\Sigma_q^\ast$ such that
\begin{equation}\label{lambdadelta.eq}
|\widehat{\mu_{q, b}}(\xi+\Pi_{\tau, \infty}(\delta\delta'))|\ge  r^{2{\mathcal D}_\tau+2} |H_{M}(\xi+\Pi_{\tau, M}(\delta))|, \ \xi\in T_b,
\end{equation}
where $r=\min\Big(r_0, \frac{1}{b}, \frac{r_1}{b(b-1)}\Big)$ and $r_0,r_1$ are defined in \eqref{r0r1.def}.
\end{lem}

\begin{proof}
If $\delta$ is a $\tau$-main branch, we
set $\delta'=0$.
In this case,
\begin{eqnarray} \label{lambdadelta.case1.eq1}
|\widehat{\mu_{q,b}}(\xi+\Pi_{\tau, \infty}(\delta\delta'))|
 &  =  & |\widehat{\mu_{q,b}}(\xi+\Pi_{\tau, M}(\delta))| \nonumber
 \\ 
 &  =  &  |H_{M}(\xi+\Pi_{\tau, M}(\delta))|
   \cdot  \big|\widehat{\mu_{q,b}}  \big((\xi+\Pi_{\tau, M}(\delta))/b^{M} \big) \big| \nonumber\\ 
 & \ge & \Big(\inf_{\eta\in (-1/(b-1), (b-2)/(b-1))}|\widehat{\mu_{q,b}}(\eta)|\Big)\cdot  |H_{M}(\xi+\Pi_{\tau, M}(\delta))|\nonumber\\
& \ge & r_0 |H_{M}(\xi+\Pi_{\tau, M}(\delta))| , \ \xi\in T_b,
\end{eqnarray}
where the second equalities follows from\eqref{eq2.c}, while
the first inequality  holds as $b^{-M}(\xi+\Pi_{\tau, M}(\delta))\in \big(-1/(b-1), (b-2)/(b-1)\big)$ for all $\xi\in T_b$.

\smallskip
Now consider $\delta$ is not a $\tau$-main branch.
In this case, define
\begin{equation}\label{lambdadelta.case2}
\delta':= 0^m\delta'',
\end{equation}
where $m\ge 1$ is the smallest integer such that  $\tau(\delta|_{m+M})\ne 0$,
and $\delta''\in \Sigma_q^\ast$ is so chosen that the quantities
$\mathcal{D}_{\tau, \delta 0^m} (\delta'')$ in \eqref{Ddelta.def0}
and ${\mathcal D}_{\tau, \delta 0^m}$ in \eqref{Ddelta.delta} are the same,
\begin{equation} \label{lambdadelta.case2.eq1}
\mathcal{D}_{\tau, \delta 0^m} (\delta'')={\mathcal D}_{\tau, \delta 0^m}.
\end{equation}
Let   $\eta_1= (\xi+\Pi_{\tau, M+m}(\delta 0^m))/b^{M+m}$ and $\eta_2=(\xi+\Pi_{\tau, M}(\delta))/b^{M}$ for $\xi\in T_b$.
Then
\begin{equation}\label{eta.def}  \eta_1\in T_b \ \  \text{and} \ \ \eta_2\in \Big(-\frac{1}{b-1}, \frac{b-2}{b-1}\Big)
\end{equation} by \eqref{tb.property} and $\tau(\delta 0^m)=\tau(\delta|_{m+M})\ne 0$.
Write
  $$\big(\Pi_{\tau,\infty}(\delta 0^m\delta'')-\Pi_{\tau,M+m}(\delta 0^m\delta'')\big)/b^{M+m}=\sum_{j=1}^K d_{n_j} b^{n_j-1}$$ for some
   integers $n_j, 1\le j\le K$, satisfying $1\le n_1<n_2<\ldots<n_K$
   and some $d_{n_j}\in \{-1, 1, 2, \ldots, b-2\}, 1\le j\le K$.
Therefore
\begin{eqnarray} \label{lambdadelta.case2.eq2}
& & \big|\widehat{\mu_{q, b}}\big(\xi+\Pi_{\tau, \infty}(\delta\delta')\big)\big| \nonumber\\
& = &
|H_{M}(\xi+\Pi_{\tau, \infty}(\delta\delta'))| \cdot \Big| \prod_{l=M+1}^{M+m} H_{q, b}\big((\xi+\Pi_{\tau, \infty}(\delta\delta'))/b^l\big)\Big|\nonumber\\
& & \quad \cdot
\big|\widehat {\mu_{q, b}}\big( (\xi+\Pi_{\tau, \infty}(\delta\delta'))/b^{M+m}\big)\big|\nonumber\\
& = & |H_{M}(\xi+\Pi_{\tau, M}(\delta))| \cdot
\Big| \prod_{l=1}^{m} H_{q, b}(\eta_2/b^l)\Big| \cdot \Big|\widehat {\mu_{q, b}}\Big( \eta_1+\sum_{j=1}^K d_j b^{n_j-1}\Big)\Big|\nonumber\\
\qquad & \ge & r_0 r^{2 \mathcal{D}_{\tau, \delta 0^m} (\delta'')}  |\widehat {\mu_{q, b}}(\eta_2)| \cdot |H_{M}(\xi+\Pi_{\tau, M}(\delta))|
\ge r_0^2 r^{2{\mathcal D}_{\tau, \delta 0^m}}  \big|H_{M}(\xi+\Pi_{\tau, M}(\delta))\big|,
\end{eqnarray}
where the first inequality follows from \eqref{eq2.c}, \eqref{Hqbproperty.one} and Lemma \ref{lem3.5}.
Combining \eqref{lambdadelta.case1.eq1} and \eqref{lambdadelta.case2.eq2} proves \eqref{lambdadelta.eq}.
\end{proof}

\medskip

Observe that   $H_{q, b}(\xi)$ 
in \eqref{hdb.def}
  satisfies
\begin{equation}\label{mqb.identity}
\sum_{j=0}^{q-1} |H_{q, b}(\xi+ j/b)|^2=1.
\end{equation}
To prove Theorem \ref{main.thm}, we  need a similar  identity for
$H_m(\xi), 
m\ge 1$, 
with shifts in $\Pi_{\tau, m}(\Sigma_q^m)$.

\begin{lem}\label{A}
Let $2\le q, b/q\in \Z$,
  $\tau:\Sigma_q^\ast\to\R$ be a tree mapping, and let
$H_m(\xi), m\ge 1$, be as in \eqref{hm.def}.  Then
\begin{eqnarray}\label{AM}
\sum_{\delta \in \Sigma_q^m} |H_m\big(\xi+\Pi_{\tau, m}(\delta)\big)|^2=1, \ \xi\in \R.
\end{eqnarray}
\end{lem}

\begin{proof}
For $m=1$,
\begin{equation*}
 \sum_{\delta \in \Sigma_q^m} |H_m(\xi+\Pi_{\tau, m}(\delta))|^2 =
  \sum_{j=0}^{q-1} |H_{q, b}(\xi/b+\tau(j)/b)|^2=
\sum_{j=0}^{q-1} |H_{q, b}(\xi/b+j/b)|^2=1,
\end{equation*}
where the last equality follows from \eqref{mqb.identity},
and the second  one holds as $H_{q, b}$ has period $q/b$
and  $\tau(j)-j\in q\Z, 0\le j\le q-1$, by the tree mapping property for $\tau$. This proves
\eqref{AM} for $m=1$.

Inductively we assume that \eqref{AM} hold for all $m\le k$. Then for $m=k+1$,
\begin{eqnarray*}
& & \sum_{\delta \in \Sigma_q^m} |H_m(\xi+\Pi_{\tau, m}(\delta))|^2\\
 & = &
 \sum_{\delta'\in \Sigma_q^k}\sum_{j=0}^{q-1}
 |H_k(\xi+\Pi_{\tau, k+1}(\delta'j))|^2
  \cdot \big|H_{q, b}\big(\xi/b^{k+1}+\Pi_{\tau, k+1}(\delta'j)/b^{k+1}\big)\big|^2\\
 & = &
 \sum_{\delta'\in \Sigma_q^k}\sum_{j=0}^{q-1}
 |H_k(\xi+\Pi_{\tau, k}(\delta'))|^2 \cdot
  \big|H_{q, b}\big(\xi/b^{k+1}+\Pi_{\tau, k}(\delta')/b^{k+1}+j/b\big)\big|^2
 =1,
  \end{eqnarray*}
where the first equality holds as $H_{k+1}(\xi)=H_k(\xi) H_{q, b}(\xi/b^{k+1})$,
the second one  follows from the observations that $H_k$ and $H_{q, b}$ are periodic functions with period $b^{k-1}q$ and $q/b$ respectively and
that $$\Pi_{\tau, k+1}(\delta'j)=
\Pi_{\tau, k}(\delta')+\tau(\delta'j) b^k\in
\Pi_{\tau, k}(\delta')+j b^k+qb^k \Z, \ 0\le j\le q-1,$$
by the tree mapping property for $\tau$, and  the last one is true by \eqref{mqb.identity} and the inductive hypothesis.
This completes the inductive proof.
\end{proof}

We have  all ingredients for the proof of Theorem \ref{main.thm}.

\begin{proof}[Proof of Theorem \ref{main.thm}]
 Let $Q(\xi):=Q_{\Lambda}(\xi)$ be the function in \eqref{Q.def}  associated with
 the maximal orthogonal set $\Lambda:=\Lambda(\tau)$
 of $L^2(\mu_{q, b})$.
As $Q$ is an analytic function on the real line, the spectral property
 for the maximal orthogonal set $\Lambda$ reduces to proving
$Q(\xi)\equiv 1$ for all $\xi\in T_b$ by  Lemma \ref{jp.lem}.
Suppose, on the contrary, there exists $\xi_0 \in T_b$ such that
\begin{equation} \label {qxi0lessone}
Q(\xi_0)<1.
\end{equation}

For $n\ge 1$, set
\begin{equation} \label {main3}
\Lambda_n:=\big\{\Pi_{\tau, \infty}(\delta): \ \delta \in \Sigma_q^n \ \text{such that} \  \tau {\rm \ is\ regular\ on\ } \delta \big\},
\end{equation}
and define
\begin{equation} \label {main4}
Q_{n}(\xi):=\sum_{\lambda \in \Lambda_{n}}|\widehat{\mu_{q,b}}(\xi+\lambda)|^2, \ \xi\in \R.
\end{equation}
Then
$$  \lim_{n \rightarrow \infty} \Lambda_n =\Lambda\ \  \text{and}\ \  \Lambda_n \subset \Lambda_{n+1}\ \text{for all} \ n\ge 1,$$
since $\Lambda=\Lambda(\tau)$  
 and
$\Sigma_q^\ast=\cup_{n=1}^\infty \Sigma_q^n$.
This implies that $Q_n(\xi), n\ge 1$, is an increasing sequence that converges to $Q(\xi)$, i.e.,
\begin{equation}\label{limitqnq}
\lim_{n\to \infty} Q_n(\xi)=Q(\xi), \ \xi\in \R.
\end{equation}
Thus for  sufficiently small $\epsilon>0$ chosen later,
there exists an integer $N$ such that
\begin{equation} \label {main4.1}
Q(\xi_0)-\varepsilon\le  Q_{N}(\xi_0)\le Q_n(\xi_0)\le Q(\xi_0)<1 \ \text{for  all} \ n\ge N.
\end{equation}

For any $\delta \in \Sigma_q^n$ being $\tau$-regular,
\begin{equation} \label {main4.2}
\lim_{m\rightarrow\infty} H_m(\xi+\Pi_{\tau, m}(\delta))= \lim_{m\rightarrow\infty} H_m(\xi+\Pi_{\tau, \infty}(\delta))= \widehat{\mu_{q,b}}(\xi+\Pi_{\tau, \infty}(\delta)), \ \xi\in \R.
\end{equation}
For any $\delta \in \Sigma_q^n$  such that  $\delta$ is not $\tau$-regular,
the set $\{m\ge n+1: \tau(\delta|_m)\ne 0\}$ contains infinite many integers.
Denote that set by $\{m_j, j\ge 1\}$ for some strictly increasing sequence $\{m_j\}_{j=1}^\infty$.
Recall that
\begin{equation}\label{taumj.property}
\tau(\delta|_{m_j})\in q\Z\cap \{-1, 1, 2,  \ldots, b-2\} \ \text{for all} \ j\ge 1
\end{equation}
 by the  tree mapping property for $\tau$.
Therefore
for $m_j\le m<m_{j+1}$ with $j\ge 1$,
\begin{eqnarray}\label{hmirregular.est}
  |H_m(\xi+\Pi_{\tau, m}(\delta))|  & \le & | H_{m_j}(\xi+\Pi_{\tau, m}(\delta))\big |
 =| H_{m_j}(\xi+\Pi_{\tau, m_j}(\delta))\big |\nonumber\\
 & \le & \prod_{k=1}^{j-1} \big | H_{q,b}\big((\xi+\Pi_{\tau, m_j}(\delta))/b^{m_k+1}\big)\big |   \nonumber\\
&  =  & \prod_{k=1}^{j-1} \big | H_{q, b}\big((\xi+\Pi_{\tau, m_k}(\delta))/b^{m_k+1}\big)\big |
\nonumber \\
 &\le & \Big (\sup_{b\eta\in T_b} |H_{q, b}(\eta)|\Big)^{j-1} ,  \ \xi\in T_b,
\end{eqnarray}
where the  inequalities  follow from \eqref{tb.property}, \eqref{Hqbproperty.one} and \eqref{taumj.property}, and  the  equalities hold by
 the tree mapping property $\tau$
 and the $q/b$ periodicity of the filter $H_{q,b}$. 
Combining \eqref{Hqbproperty.one} and \eqref{hmirregular.est} proves that
\begin{equation}\label {main4.3}
\lim_{m\to \infty} |H_m(\xi+\Pi_{\tau, m}(\delta))|=0,\  \xi\in T_b
\end{equation}
if $\delta\in \Sigma_q^n$ is not $\tau$-regular.

Applying \eqref{main4.2} and \eqref{main4.3} with $n$ and $\xi$ replaced by $N$ and $\xi_0$
respectively, we can find a sufficient large integer $M\ge N+1$
 such that
\begin{equation}\label{n1.eq1}
\sum_{\delta \in \Sigma_q^{N}} | H_{M}(\xi_0+\Pi_{\tau, M}(\delta))|^2 \le
\sum_{\lambda \in \Lambda_{N}}|\widehat{\mu_{q,b}}(\xi_0+\lambda)|^2 + \varepsilon \le Q(\xi_0)+\varepsilon .
\end{equation}
This together with Lemma \ref{A} implies that
\begin{equation} \label {main5}
\sum_{\delta \in \Sigma_q^{M}\backslash \Sigma_q^{N}}
 | H_{M}(\xi_0+\Pi_{\tau, M}(\delta))|^2 >1- Q(\xi_0)-\varepsilon>0,
\end{equation}
where
$$\Sigma_q^{M}\backslash \Sigma_q^{N}=\big\{\delta\in \Sigma_q^{M}: \   \delta|_N0^\infty\ne \delta0^\infty\big\}.$$

Now, for each $\delta \in \Sigma_q^{M}\backslash \Sigma_q^{N}$, let $\lambda(\delta)=\Pi_{\tau, \infty}(\delta\delta')$ with $\delta'$ selected as in Lemma \ref {Add}. Observe that  $\lambda(\delta)-\Pi_{\tau, M}(\delta)\in b^M\Z$ for all $\delta\in \Sigma_q^M\backslash \Sigma_q^N$.
This implies that $\lambda(\delta_1)\ne \lambda(\delta_2)$ for two distinct $\delta_1, \delta_2\in \Sigma_q^M\backslash \Sigma_q^N$.
Therefore
\begin{eqnarray} \nonumber
Q(\xi_0)=\sum_{\lambda \in \Lambda}|\widehat{\mu_{q,b}}(\xi_0+\lambda)|^2
& \ge & \sum_{\lambda \in \Lambda_{N}}|\widehat{\mu_{q,b}}(\xi_0+\lambda)|^2 + \sum_{ \delta \in \Sigma_q^{M}\backslash \Sigma_q^{N}}|\widehat{\mu_{q,b}}(\xi_0+\lambda(\delta))|^2  \\ \nonumber
& \ge & Q(\xi_0)-\varepsilon+  r^{4{\mathcal D}_\tau+4}\sum_{ \delta \in \Sigma_q^{M}\backslash \Sigma_q^{N}}|H_M(\xi+\Pi_{\tau, M}(\delta))|^2 \\ \nonumber
& \ge & Q(\xi_0)-\varepsilon+  r^{4{\mathcal D}_\tau+4} (1- Q(\xi_0)-\varepsilon),
\end{eqnarray}
where the second inequality follows from \eqref{main4.1} and Lemma \ref{Add}, and
the last holds by \eqref{main5}. This contradicts to \eqref {qxi0lessone} by letting $\varepsilon$ chosen sufficiently small.
\end{proof}

\section{Maximal orthogonal sets and spectra: a necessary condition}
\label{necesarry.section}

Given a tree mapping $\tau$,  define
\begin{equation}
\label{Nn.def}
{\mathcal N}_{\tau}(n):=
\left\{\begin{array}{ll}
\inf_{0\ne \delta\in \Sigma_q}
{\mathcal D}_{\tau, \delta} (0^\infty) & {\rm if} \ n=1\\
\inf_{\delta\in \Sigma_q^n\backslash \Sigma_q^{n-1}}
{\mathcal D}_{\tau, \delta} (0^\infty) & {\rm if}\    n\ge 2,
\end{array}\right. \end{equation}
where $\Sigma_q^n\backslash \Sigma_q^{n-1}:=\{\delta'j: \ \delta'\in \Sigma_q^{n-1}, 1\le j\le q-1\}$.
In this section, 
 we   establish the following strong version of
  Theorem \ref{main2.thm.old}.

\begin{theo}\label{main2.newthm}
Let $2\le q, b/q\in \Z$, $\tau: \Sigma_q^\ast\to \{-1, 0, \ldots, b-2\}$ be a maximal tree mapping, and
let ${\mathcal N}_\tau(n), n\ge 1$, be as in \eqref{Nn.def}. Set
$$r_2:=\max\{|H_{q, b}(\xi)|: 1/b\le b(b-1) |\xi|\le b-2\}.$$
 If $\sum_{n=1}^\infty r_2^{2\mathcal{N}_\tau(n)}<\infty$, then $\Lambda(\tau)$  in \eqref{lambdatau.def} is not a spectrum of $L^2(\mu_{q, b})$.
\end{theo}

For a maximal tree mapping $\tau$ satisfying \eqref{main2.thm.old.eq1},
$$\sum_{n=1}^\infty r_2^{2\mathcal{N}_\tau(n)}\le \sum_{n=1}^\infty r_2^{2\epsilon_0 n}<\infty,$$
where the last inequality holds as $|H_{q, b}(\xi)|<1$ if $b\xi\notin q\Z$.
This together with Theorem \ref{main2.newthm} proves Theorem \ref{main2.thm.old}.
Now it remains to prove Theorem \ref{main2.newthm}.

\begin{proof}  [Proof of Theorem \ref{main2.newthm}]
Let $N_0\ge 2$ be so chosen that ${\mathcal N}_\tau(n)\ge 1$ for all $n\ge N_0$. The existence follows the series convergence assumption on ${\mathcal N}_\tau(n), n\ge 1$.
Take   $\delta \in \Sigma_q^{n} \backslash \Sigma_q^{n-1}$
being $\tau$-regular, where $n
\ge N_0$. Write
$$\{m\ge n+1: \ \tau(\delta|_m)\ne 0\}=\{n_k: 1\le k\le K\}$$
for some integers $n< n_1<n_2<\ldots <n_{K}$, where $K\ge {\mathcal N}_\tau(n)$.
Therefore for $\xi \in T_b$,
  \begin{eqnarray}\label{main2.3}
|\widehat{\mu_{q,b}}(\xi+\Pi_{\tau, \infty}(\delta))|
& = & |H_n(\xi+\Pi_{\tau, \infty}(\delta))| \cdot \big|\widehat{\mu_{q,b}}\big((\xi+\Pi_{\tau, \infty}(\delta))/b^n\big)\big|\nonumber\\
& \le &   |H_n(\xi+\Pi_{\tau, n}(\delta))| \cdot \prod_{k=1}^{K}
\big|H_{q, b}\big((\xi+\Pi_{\tau, n_{k}}(\delta))/b^{-n_k-1}\big)\big| \nonumber\\
&\le  &  \Big(\sup_{\eta\in T_b} |H_{q, b}(\eta/b)|\Big)^K \cdot
 |H_n(\xi+\Pi_{\tau, n}(\delta))|\nonumber\\
& \le & r_2^{{\mathcal N}_\tau(n)}
 |H_n(\xi+\Pi_{\tau, n}(\delta))|,
\end{eqnarray}
where the first equality  holds by \eqref{eq2.c};
 the first inequality follows from \eqref{eq2.c}, \eqref{Hqbproperty.one} and
 $\tau(\delta|_{n_{k}})\in q\Z, 1\le k\le K$, by the tree mapping property for $\tau$;
the second inequality is true since $(\xi+\Pi_{\tau, n_{k}}(\delta))/b^{-n_k}\in T_b$ by \eqref{tb.property};
and the last inequality  follows from the definition of the quality ${\mathcal N}_\tau(n)$.

Let $\Lambda_n$ and $Q_n, n\ge 1$, be as in \eqref{main3} and \eqref{main4} respectively,
and set $\Lambda_0=\{0\}$ and $Q_0(\xi)=|\widehat{\mu_{q,b}}(\xi)|^2$.
Then for $n\ge 1$ and  $\xi\in T_b$,
\begin{eqnarray}\label{necessary.prof.main}  
1-Q_{n}(\xi)
&  =  & 1-Q_{n-1} (\xi)-\sum_{\delta \in \Sigma_q^{n}\backslash\Sigma_q^{n-1}\ \text{is $\tau$-regular}}|\widehat{\mu_{q,b}}(\xi+\Pi_{\tau, \infty}(\delta))|^2 \\ \nonumber
& \ge & 1-Q_{n-1} (\xi)- r_2^{2 \mathcal{N}_{\tau}(n)}  \sum_{\delta \in \Sigma_q^{n}\backslash\Sigma_q^{n-1}}| H_{n}(\xi+\Pi_{\tau, n}(\delta))|^2 \\ \nonumber
& \ge & 1-Q_{n-1} (\xi)- r_2^{2 \mathcal{N}_{\tau}(n)}
\Big(1- \sum_{\lambda\in \Lambda_{n-1}}| \widehat{\mu_{q,b}}(\xi+\lambda)|^2 \Big) \\ \nonumber
&  =  & (1- r_2^{2 \mathcal{N}_{\tau}(n)} )\cdot\big(1-Q_{n-1}(\xi)\big),
\end{eqnarray}
where the first equality holds because $$\Lambda_{n}\backslash\Lambda_{n-1}=
\big\{\Pi_{\tau, \infty}(\delta): \ \delta\in \Sigma_q^n\backslash \Sigma_q^{n-1} \ \text{is $\tau$-regular}\big\};$$
the first inequality  is true by \eqref{main2.3};  and the second inequality  follows from
Lemma \ref{A} and
\begin{equation*}
\sum_{\lambda\in \Lambda_{n-1}}| \widehat{\mu_{q,b}}(\xi+\lambda)|^2 \le \sum_{\delta\in \Sigma_q^{n-1}}| H_{n}(\xi+\Pi_{\tau, n}(\delta))|^2, \ \ \xi\in \R,
\end{equation*}
by \eqref{eq2.c} and \eqref{Hqbproperty.one}.
Recall that $\lim_{n\to \infty} Q_n(\xi)=Q(\xi), \xi\in \R$, by \eqref{limitqnq}.
Applying \eqref{necessary.prof.main} repeatedly and using the convergence of $\sum_{n=1}^\infty r_2^{2{\mathcal N}_\tau(n)}$ gives
\begin{equation}\label{almostfornecessary}
1-Q(\xi)\ge \Big(\prod_{n=N_0+1}^\infty \big(1-r_2^{2 {\mathcal N}_\tau(n)}\big)\Big)\cdot (1-Q_{N_0}(\xi)),\  \  \xi\in T_b.
\end{equation}
On the other hand,
$$Q_{N_0}(\xi)=\sum_{\lambda\in \Lambda_{N_0}}
|\widehat{\mu_{q,b}}(\xi+\lambda)|^2<\sum_{\delta\in \Sigma_q^{N_0}} |H_{N_0}(\xi+\Pi_{\tau, {N_0}}(\delta))|^2=1, \ \ \xi\in T_b$$
by \eqref{eq2.c}, \eqref{Hqbproperty.one} and \eqref{AM}.
This together with \eqref{almostfornecessary} proves that
$Q(\xi)<1$ for all $\xi\in T_b$, and hence $\Lambda=\Lambda(\tau)$ is a not a spectrum for $L^2(\mu_{q, b})$ by
Lemma \ref{jp.lem}.
\end{proof}

\section{Spectra rescaling}
\label{rescaledspectra.section}

In this section, we first prove Theorem \ref{multiple.thm} in Subsection \ref{multiple.thm.proofsection}.
We then consider  verification of  maximal orthogonality of the rescaled set $K\Lambda$ in Subsection \ref{rescaledmos.subsection}.
In that subsection, we
show that the  rescaled set $K\Lambda$ is not a maximal orthogonal set of the Cantor measure
$\mu_{q, b}$ if and only if  the labeling tree $\tau(\Sigma_q^\ast)$ has certain periodic properties \eqref{app.periodic.eq} and \eqref{app0.1}.

\begin{theo}\label{App.thm}
Let $2\le q, b/q\in \Z$,  $\tau: \Sigma_q^\ast\to \{-1, 0, \ldots, b-2\}$ be a maximal tree mapping,
 $\Lambda:=\Lambda(\tau)$  be as in \eqref{lambdatau.def}, and  let $K>1$ be an integer coprime with $b$.
Then $K\Lambda$ is not a  maximal orthogonal set of the Cantor measure $\mu_{q,b}$ if and only if
there exist $\delta \in \Sigma_q^\infty$ and a nonnegative integer $M$  such that
$\{\tau(\delta|_n)\}_{n=M+1}^\infty$ is a periodic sequence with positive period $N$, i.e.,
\begin{equation}\label{app.periodic.eq}
\tau(\delta|_n)=\tau (\delta|_{n+N}), \ n\ge M+1, \end{equation}
and that
the word
 $W=\omega_1\omega_2\cdots\omega_{N}$ defined by
 \begin{equation}\label{omegaword.def} \omega_j=\tau(\delta|_{M+j}), 1\le j\le N,
\end{equation}
is a {\em repetend} of the recurring $b$-band decimal expression of $i/K$ for some $i\in \Z\backslash \{0\}$, i.e.,
\begin{equation}\label{app0.1}
\frac{i}{K}=0.\omega_N\cdots \omega_2\omega_1\omega_N\cdots \omega_2\omega_1\omega_N\cdots= \sum_{n=1}^\infty \sum_{j=1}^{N}\omega_jb^{j-Nn-1}=\frac{\sum_{j=1}^{N}\omega_jb^{j-1}}{b^N-1}.
\end{equation}
\end{theo}

By Theorems \ref{multiple.thm} and \ref{App.thm}, we see that
 the rescaled set $K\Lambda$ is a spectrum if and only if  the labeling tree of $\Lambda$ contains no \emph{repetend} of $K$.

For the  spectrum $\Lambda_4$
 of the Bernoulli convolution $\mu_4$
 in \eqref{lambda4.def}, the associated maximal tree mapping $\tau_{2, 4}$
on $\Sigma_2^\ast$ is given by
\begin{equation}\label{tau24.def}
\tau_{2, 4}(\delta)=\delta_n \ \ {\rm for }\ \ \delta=\delta_1\cdots \delta_n\in \Sigma_2^n,\ n\ge 1.
\end{equation}
Thus
$\mathcal{D}_{\tau_{2, 4}, \delta}=0$  for  all $\delta\in \Sigma_2^\ast$, and
the requirement \eqref{main.thm.eq1} is satisfied for the maximal tree mapping $\tau_{2, 4}$.
Hence
Corollary \ref{cantor24.cor} follows immediately from
Theorem \ref{multiple.thm} and \ref{App.thm}. 

Finally in Subsection \ref{example.section}, we construct a spectrum $\Lambda$
 of the Cantor measure $\mu_{q, b}$
 such that $\Lambda/(b-1)$, a seemingly denser set than the spectrum $\Lambda$,  is
 its maximal orthogonal set but not
 its spectrum.

\begin{theo}\label{counterexample}
Consider $2\le q, b/q\in \Z$ and $b>4$.  Define a tree mapping $\kappa: \Sigma_q^{\ast} \rightarrow \{-1,0,1,\ldots, b-2\}$ by
 \begin{equation}\label{counterexample.eq1}
  \kappa(\delta|_{k+1})=\left\{\begin{array}{ll}0 & {\rm if} \ \delta=0\ {\rm and} \ k\ge 0\\
  \delta & {\rm if} \ 1\le \delta\le q-1\ {\rm and} \ k=0\\
 q & {\rm if} \ 1\le \delta\le q-1\ {\rm and} \ k\in \{1, 2, \cdots, K_\delta, 2b\}\\
 0 & {\rm if} \ 1\le \delta \le q-1\ {\rm and} \ K_\delta< k\ne 2b
 \end{array}\right. \quad {\rm if} \ \ \delta\in \Sigma_q^1,
 \end{equation}
 where $0\le K_\delta\le b-2$ is the unique integer
 such that $q(K_\delta+1)+\delta\in (b-1)\Z$; and inductively
 \begin{equation} \label{counterexample.eq2}
 \kappa(\delta|_{k+n}) =\left\{\begin{array}{ll} j & {\rm if} \ k=0\\
 q & {\rm if} \ k\in \{1, 2, \ldots, K_\delta, n+2b-1\}\\
 0 & {\rm if} \ k>K_\delta \ {\rm and}\ k\ne n+2b-1
 \end{array}\right.
  \end{equation}
 if $\delta=\delta'j$ for some $\delta'\in \Sigma_q^{n-1}, n\ge 2$ and $j\in \{1, \ldots, q-1\}$, where $K_\delta\in \{0, 1, \ldots, b-2\}$
 is the unique integer such that
 \begin{equation} \label{counterexample.eq3}
\Big (\sum_{i=1}^{n-1}\kappa(\delta|_i)+q(K_\delta+1)+j\Big)\in (b-1)\Z. \end{equation}
Then \begin{equation}\label{lambda210.def}
\Lambda_{q, b}:=\big\{\Pi_{\kappa, \infty}(\delta): \ \delta\in \Sigma_q^\ast\big\}
\end{equation}
is a spectrum of the Cantor measure $\mu_{q, b}$,
and the rationally rescaled set $\Lambda_{q, b}/(b-1)$
is its maximal orthogonal set
   but not its spectrum.
\end{theo}


%
%
%

\subsection{Proof of Theorem \ref{multiple.thm}}
\label{multiple.thm.proofsection}
%
%
The necessity  is obvious. Now we prove  the sufficiency.
%
%
Without loss of generality, we assume $K$ is positive since $-\Lambda$ is a spectrum (maximal orthogonal set) if and only if $\Lambda$ is. Let $\kappa$ be the maximal tree mapping associated with the maximal orthogonal set $K\Lambda$ of the Cantor measure $\mu_{q, b}$.
The existence of such a mapping follows from Theorem \ref{th1.6} and the assumption on $K\Lambda$.
Denote the integral part of a real number $x$ by $\lfloor x \rfloor$.
By Theorem \ref{main.thm} and the assumption that ${\mathcal D}_\tau<\infty$,
it suffices to prove that
\begin{equation}\label{dkappadtau.eq0}
\inf \{{\mathcal D}_{\kappa,\delta}(\delta'), \ \delta'\in \Sigma_q^\ast\}
\le (2\lfloor \log_b K\rfloor+4) ({\mathcal D}_{\tau}+1), \ \delta\in \Sigma_q^\ast.
\end{equation}

Take $\delta\in \Sigma_q^n, n\ge 1$, and let $\delta_1\in \Sigma_q^\ast$ be so chosen that
$\delta\delta_1$ is $\kappa$-regular. 
As
$\Pi_{\kappa, \infty}(\delta\delta_1)\in K\Lambda$, there exists $\zeta\in \Sigma_q^n$ such that
\begin{equation}\label{multiple.thm.pf.eq1} K\Pi_{\tau, n}(\zeta)-\Pi_{\kappa, n} (\delta)
\in b^n\Z.\end{equation}
Let  $\zeta'\in \Sigma_q^\ast$  be so chosen that
$\zeta\zeta'$ is a  $\tau$-main subbranch of $\zeta$ and
\begin{equation}
\mathcal{D}_{\tau, \zeta} (\zeta')=  \mathcal{D}_{\tau, \zeta},\end{equation}
where  the existence of such a tree branch $\zeta'$  follows from
\eqref{Ddelta.delta2}.
Therefore the verification of
\eqref{dkappadtau.eq0} reduces to showing the existence of $\delta'\in \Sigma_q^\ast$ such that
$\delta\delta'$ is a $\kappa$-main branch,
\begin{equation}\label{zetazeta.eq}
K \Pi_{\tau, \infty}(\zeta\zeta')= \Pi_{\kappa, \infty} (\delta\delta'),
\end{equation}
 and
\begin{equation}\label{dkappadtau.eq}
{\mathcal D}_{\kappa,\delta}(\delta')\le (2\lfloor \log_b K\rfloor+4) ({\mathcal D}_{\tau, \zeta}+1).
\end{equation}

By Theorem \ref{th1.6}, there exists a $\kappa$-main branch $\delta_2\in \Sigma_q^\ast$ such that
\begin{equation} \label{zetazeta.eq22}
\Pi_{\kappa, \infty} (\delta_2)=K \Pi_{\tau, \infty}(\zeta\zeta').\end{equation}
Then
\begin{equation}
\Pi_{\kappa, n} (\delta_2)-\Pi_{\kappa, n}(\delta)\in K \Pi_{\tau, n}(\zeta)-\Pi_{\kappa, n}(\delta)+b^n\Z=b^n\Z
\end{equation}
by \eqref{multiple.thm.pf.eq1}. This together with one-to-one correspondence of the mapping $\Pi_{\kappa, n}:\Sigma_q^n\to\Z$
proves $\delta_2=\delta\delta'$ for some $\delta'\in \Sigma_q^\ast$.

The equation \eqref{zetazeta.eq} follow from \eqref{zetazeta.eq22}. Now it remains to prove
\eqref{dkappadtau.eq}.
Without loss of generality, we assume that
 $\Pi_{\kappa, \infty} (\delta\delta')
\ne \Pi_{\kappa, n}(\delta)$, because otherwise ${\mathcal D}_{\kappa,\delta}(\delta')=0$ and hence
\eqref{dkappadtau.eq} follows immediately.
 Thus we may write
\begin{equation}\label{th2.6.pf.eq4}
 \Pi_{\kappa, \infty} (\delta\delta')
=\Pi_{\kappa, n}(\delta)+\sum_{l=1}^L d_l b^{n+m_l-1}
\end{equation}
for a strictly increasing sequence $\{m_l\}_{l=1}^L$ of integers
and some  $d_l\in \{-1,  1, \ldots, b-2\}, 1\le l\le L$.

Also we may assume that
 $\Pi_{\tau, \infty}(\zeta\zeta')\ne \Pi_{\tau, n}(\zeta)$, because otherwise
$$K\Pi_{\tau, \infty}(\zeta\zeta') =K\Pi_{\tau, n}(\zeta)
\in 
 K(-b^n/(b-1), (b-2)b^n/(b-1))$$
and
\begin{equation*}
\Pi_{\kappa, \infty} (\delta\delta')
\not \in (-b^{n+m_L-1}/(b-1), (b-2) b^{n+m_L-1}/(b-1))
\end{equation*}
by \eqref{tb.property} and \eqref{th2.6.pf.eq4}.
This together with \eqref{zetazeta.eq} implies that
$b^{m_L-1}\le K$ and hence
 $${\mathcal D}_{\kappa,\delta}(\delta')\le m_L\le \lfloor\log_b K\rfloor+1.$$
Therefore we can write
$$
\Pi_{\tau, \infty}(\zeta\zeta')=\Pi_{\tau, n}(\zeta)+\sum_{j=1}^N  c_j b^{n+n_j-1},
$$
where $c_j\in \{-1, 1, \ldots, b-2\}, 1\le j\le N$,
and $\{n_j\}_{j=1}^N$ is a strictly increasing sequence of integers.

To prove \eqref{dkappadtau.eq} for  the case that $\Pi_{\tau, \infty}(\zeta\zeta')\ne \Pi_{\tau, n}(\zeta)$, we need the following claim:

\smallskip

{\em Claim 1: 
$\{m_l,1 \le l\le L\}\subset \cup_{j=0}^N [n_j, n_j+\lfloor\log_b K\rfloor+1]$.
}
\begin{proof}
Suppose, on the contrary, that Claim 1 does not hold. Then
there exists $1\le l\le L$ such that $n_{j_0}+\lfloor\log_b K\rfloor+1<m_l<n_{j_0+1}$ for some $0\le j_0\le N$, where we set $n_0=0$ and $n_{N+1}=+\infty$.
Observe that
\begin{equation} \label{mlnj.inclusion.pf1}
  \Pi_{\kappa, n+m_l}(\delta\delta')- K\Pi_{\tau, n+n_{j_0}}(\zeta\zeta')\in b^{n+m_l}\Z
  \end{equation}
by \eqref{zetazeta.eq} and the assumption $m_l<n_{j_0+1}$, and
\begin{eqnarray} \label{mlnj.inclusion.pf2}
& &   \Pi_{\kappa, n+m_l}(\delta\delta')- K\Pi_{\tau, n+n_{j_0}}(\zeta\zeta')\nonumber\\
 & \in  & d_l b^{n+m_l-1}+ \frac{b^{n+m_l-1}-1}{b-1}[-1, b-2]-
  K \frac{b^{n+n_{j_0}}-1}{b-1} [-1, b-2]\nonumber\\
&    \subset &  d_l b^{n+m_l-1}+(-b^{n+m_l-1}, b^{n+m_l-1})
\end{eqnarray}
by the definitions of $\Pi_{\kappa, n+m_l}$ and $\Pi_{\tau,n+ n_{j_0}}$
and the assumption $n_{j_0}+\log_b K+1<m_l$.
Combining \eqref{mlnj.inclusion.pf1} and \eqref{mlnj.inclusion.pf2} leads to the contradiction that $d_l\in \{-1, 1, \ldots, b-2\}$.
This completes the proof of Claim 1.
\end{proof}

To prove \eqref{dkappadtau.eq} for the case that $\Pi_{\tau, \infty}(\zeta\zeta')\ne \Pi_{\tau, n}(\zeta)$, we need another claim:

\smallskip
{\em Claim 2: If $n_{j}+\lfloor \log_b K\rfloor+1 <n_{j+1}$, then there exists $l_0$ such that $m_{l_0}=n_{j+1}$, $m_{l_0-1}\in [n_j, n_j+\lfloor \log_b K\rfloor+1]$
and $d_{l_0}\in q\Z$ if and only if $c_{j+1}\in q\Z$.
}
\begin{proof}
Let $l_0$ be the smallest integer $l$ with $m_l \ge n_{j+1}$.
By Claim 1, $m_{l_0-1}\le n_j+\lfloor \log_b K\rfloor+1\le n_{j+1}-1$.
Observe that
$  \Pi_{\kappa, n+m_{l_0}}(\delta\delta')- K\Pi_{\tau, n+n_{j+1}}(\zeta\zeta')
\in   b^{n+n_{j+1}}\Z$  by \eqref{zetazeta.eq}; and
  \begin{eqnarray}\label{pipinew.eq}
& &   \Pi_{\kappa, n+ m_{l_0}}(\delta\delta')- K\Pi_{\tau, n+n_{j+1}}(\zeta\zeta')\nonumber\\
& \in &  d_{l_0} b^{n+m_{l_0}-1}-Kc_{j+1} b^{n+n_{j+1}-1}+\frac{ b^{n+m_{l_0-1}}}{b-1}
 (-1,  b-2)-\frac{K b^{n+n_{j}}}{b-1}(-1, b-2)\nonumber\\
\qquad & \subset &
d_{l_0} b^{n+m_{l_0}-1}-Kc_{j+1} b^{n+n_{j+1}-1}+b^{n+n_{j+1}-1} (-1, 1).
\end{eqnarray}
Thus
$d_{l_0} b^{m_{l_0}-n_{j+1}}-Kc_{j+1}\in b\Z$.
This together, with the assumptions that $c_{j+1}\in \{-1, 1, \ldots, b-2\}$ and
that $K$ and $b$ are coprime, implies that $m_{l_0}=n_{j+1}$
and $d_{l_0}\in q\Z$ if and only if $c_{j+1}\in q\Z$.
From the argument in \eqref{pipinew.eq}, we see that
\begin{equation}
\Pi_{\kappa, m_{l_0-1}}(\delta\delta')=K\Pi_{\tau, n_{j}}(\zeta\zeta').
\end{equation}
Thus $m_{l_0-1}\ge n_j$, as
$\Pi_{\kappa, m_{l_0-1}}(\delta\delta') \in b^{m_{l_0-1}} \big(-1/(b-1), (b-2)/(b-1)\big)$ and
$K\Pi_{\tau, n_{j}}(\zeta\zeta')\not\in K b^{n_j-1} \big(-1/(b-1), (b-2)/(b-1)\big)$ by \eqref{tb.property}.
This completes the proof of Claim 2.
\end{proof}

\smallskip

Having established the above two claims, let us return to the proof of the inequality \eqref{dkappadtau.eq}.
Note that if $$\{k\in\Z: m_{l_0-1}<k<m_{l_0}\} \not\subset \cup_{j=0}^N [n_j, n_j+\lfloor\log_b K\rfloor+1]$$
for some $1\le l_0\le L$, then by Claim 1, there exists $1\le j_0\le N$ such that $$m_{l_0-1} \le n_{j_0-1}+\lfloor\log_b K\rfloor+1<n_{j_0}\le m_{l_0}.$$
Then $m_{l_0}=n_{j_0}$, $m_{l_0-1}\ge n_{j_0-1}$ and $d_{l_0}\in q\Z$ if and only if $c_{j_0}\in q\Z$ by Claim 2. Thus
$$
\cup_{d_l\not\in q\Z}(m_{l-1},m_{l}) \subset \left(\cup_{j=0}^N [n_j, n_j+\lfloor\log_b K\rfloor+1]\right) \cup \left(\cup_{c_j\not\in q\Z} (n_{j-1},n_j)\right),
$$
and thus
$$
\sum_{d_l\not\in q\Z} (m_l-m_{l-1}-1)
\le (\lfloor\log_b K\rfloor+2) (N+1)+\sum_{c_j\not\in q\Z} (n_j-n_{j-1}-1).
$$
This together with Claim 1, implies
\begin{equation*}
{\mathcal D}_{\kappa,\delta}(\delta')
\le  2(\lfloor\log_b K\rfloor+2) (N+1)+ \sum_{c_j\not\in q\Z} (n_j-n_{j-1}-1)
\le (2\lfloor\log_b K\rfloor+4) ({\mathcal D}_{\tau,\zeta}(\zeta') +1).
\end{equation*}
We get \eqref{dkappadtau.eq} and hence complete the proof of Theorem \ref{multiple.thm}.

%

\subsection{Proof of Theorem \ref{App.thm}} 
\label{rescaledmos.subsection}
$(\Longleftarrow)$\quad
Let
\begin{equation} \label{App.thm.pf.eq3}
\lambda_0=K \Pi_{\tau, M}(\delta)- i b^M,\end{equation}
where  $i\in \Z$ is given  in \eqref{app0.1}.
Inductively applying \eqref{app0.1} proves that
\begin{equation}\label{App.thm.pf.eq4}
 \lambda_0
 =  K \Pi_{\tau, M+N}(\delta)- i b^{M+N}
 =\cdots=
   K \Pi_{\tau, M+nN}(\delta)-i b^{M+nN}, \ n\ge 1.
\end{equation}

Take $\lambda\in \Lambda$.
Now we show that
$\exp({-2\pi i\lambda_0x})$ is orthogonal to $\exp({-2\pi iK\lambda x})$.
By the maximality of the tree mapping $\tau$,
there exists a $\tau$-main branch  $\zeta\in \Sigma_q^m$ for some $m\ge 1$
by Theorem \ref{th1.6} such that
\begin{equation} \label{App.thm.pf.eq2}
\lambda=\Pi_{\tau, \infty}(\zeta).\end{equation}

Also for sufficiently large $n\ge 1$,
 there exists $\lambda_n\in\Lambda$  by
 the maximality of the tree mapping $\tau$
 such that $\lambda_n\ne \lambda$ and
  \begin{equation} \label{App.thm.pf.eq5}
  \lambda_n-\Pi_{\tau, M+Nn}(\delta)\in b^{M+Nn}\Z.
  \end{equation}
  The reason
   for  $\lambda_n\ne \lambda$
is that 
    $\Pi_{\tau, M+Nn}(\delta)\ne \Pi_{\tau, M+Nn}(\zeta)$
  for sufficiently large $n$  by $W=\omega_1\ldots \omega_N\ne 0^N$ by
  \eqref{app0.1}.

  As both $\lambda, \lambda_n\in \Lambda$, there exists a
  nonnegative integer $l$ and an integer $a\in \Z\backslash q\Z$ by \eqref{orthogonalcharacterization} such that
  \begin{equation} \label{App.thm.pf.eq6}
  \lambda-\lambda_n= ab^l.
  \end{equation}
  Now we show that
  \begin{equation} \label{App.thm.pf.eq7} l<M+Nn\end{equation}
   when $n$ is sufficiently large.
  Suppose, on the contrary, that $l\ge M+Nn$. Then
  \begin{equation} \label{App.thm.pf.eq8}
  \lambda-\Pi_{\tau, M+Nn}(\delta)\in b^{M+Nn}\Z.
  \end{equation}
  On the other hand,
  \begin{equation*}
  \Pi_{\tau, M+Nn}(\delta)\in b^{M+Nn} [-1/(b-1), (b-2)/(b-1)]\end{equation*}
by the tree mapping property for $\tau$.
Therefore
$\lambda=\Pi_{\tau, M+Nn}(\delta)$
for sufficiently large $n$, which is a contradiction as
  \begin{equation*}
  \Pi_{\tau, M+Nn}(\delta)\not\in b^{M+N(n-1)} (-1/(b-1), (b-2)/(b-1))\end{equation*}
    by $W=\omega_1\ldots \omega_N\ne 0^N$ and the tree mapping property for $\tau$.

  Combining \eqref{App.thm.pf.eq5},
  \eqref{App.thm.pf.eq6} and \eqref{App.thm.pf.eq7} and recalling that $K$ and $b$ are co-prime, we obtain that
  \begin{equation} \label{App.thm.pf.eq9}
 K \lambda- K\Pi_{\tau, M+Mn}(\delta)=\tilde a b^l
  \end{equation}
  for some integers $0\le l<M+Nn$ and $\tilde a\in\Z\backslash q\Z$.
  Thus the inner product between $\exp({-2\pi i\lambda_0x})$ and $\exp({-2\pi iK\lambda x})$ is equal to zero by
  \eqref{orthogonalcharacterization},
   \eqref{App.thm.pf.eq4} and \eqref{App.thm.pf.eq9}.
   This proves that $K\Lambda$ is not a maximal orthogonal set as
    $\lambda\in \Lambda$ is chosen arbitrarily. 

  \bigskip

($\Longrightarrow$)\quad 
  By \eqref{orthogonalcharacterization}
  and the assumption on the rescaled set $K\Lambda$, there exists
   a maximal orthogonal set $\Theta$ of the Cantor measure $\mu_{q, b}$ such that
   \begin{equation}\label{App.thm.pf.part2eq1}
   K\Lambda\varsubsetneq \Theta\subset \Z.
   \end{equation}
 Take $\vartheta_0\in \Theta\backslash (K\Lambda)$.
 Then
 \begin{equation} \label{App.thm.pf.part2eq2}
 \vartheta_0=\Pi_{\kappa, \infty}(\zeta_0)=\Pi_{\kappa, m}(\zeta_0)
 \end{equation}
 for some $\kappa$-main branch $\zeta_0\in \Sigma_q^m, m\ge 1$, where
   $\kappa$ is the maximal tree mapping associated with the maximal orthogonal set $\Theta$. 

 Let $\tau$ be the maximal tree mapping in Theorem \ref{th1.6}
 such that $\Lambda=\Lambda(\tau)$.
To establish the necessity, we need the following claim:

 {\em Claim 3: Let $n\ge 1$. For any $\zeta\in \Sigma_q^n$ there exists a unique $\delta\in \Sigma_q^n$ such that $\Pi_{\kappa, n}(\zeta)-K \Pi_{\tau, n}(\delta)\in b^n\Z$}.

\begin{proof}
  Observe that
  \begin{equation} \label{App.thm.pf.part2eq3}
   K\Pi_{\tau, n}(\delta_1)-K\Pi_{\tau, n}(\delta_2)\not\in b^n\Z
  \quad {\rm for \ all \ distinct} \ \delta_1, \delta_2\in \Sigma_q^n,\end{equation}
because
 $b/q\in \Z$,  $K$ and $b$ are coprime, and
$\Pi_{\tau, n}(\delta_1)-\Pi_{\tau, n}(\delta_2)=a b^l$
 for some $0\le l\le n-1$ and $a\not\in q\Z$.
On the other hand,
\begin{equation} \label{App.thm.pf.part2eq4}
\{K\Pi_{\tau, n}(\delta): \ \delta\in \Sigma_q^n\}+b^n\Z= K\Lambda+b^n\Z
\subset \Theta+b^n\Z=
\{\Pi_{\kappa, n}(\zeta): \ \zeta\in \Sigma_q^n\}+b^n\Z\end{equation}
by \eqref{App.thm.pf.part2eq1}.
Combining \eqref{App.thm.pf.part2eq3} and \eqref{App.thm.pf.part2eq4} leads to
\begin{equation} \label{App.thm.pf.part2eq5}
\{K\Pi_{\tau, n}(\delta): \ \delta\in \Sigma_q^n\}+b^n\Z=
\{\Pi_{\kappa, n}(\zeta): \ \zeta\in \Sigma_q^n\}+b^n\Z.\end{equation}
 Then Claim 3 follows from \eqref{App.thm.pf.part2eq5} and \eqref{App.thm.pf.part2eq3}.
\end{proof}

To establish the necessity, we need another claim:

{\em Claim 4: $\vartheta_0\not \in K\Z$.}

\begin{proof} Suppose, on the contrary, that $\vartheta_0\in K\Z$.
Then for any $\lambda\in \Lambda$,
there exist $a\in \Z\backslash q\Z$ and $0\le l\in \Z$ by \eqref{orthogonalcharacterization} and \eqref{App.thm.pf.part2eq1}
such that $\vartheta_0-K\lambda=a b^l$. This together with
the co-prime assumption between $K$ and $b$ implies that $a/K\in \Z$ and
$0\ne \vartheta_0/K-\lambda\in (a/K) b^l$. Thus $\Lambda\cup\{\vartheta_0/K\}$
is an orthogonal set for the measure $\mu_{q, b}$ by \eqref{orthogonalcharacterization}, which contradicts to the maximality of the set $\Lambda$.
\end{proof}


Now we continue our proof of the necessity.
Let $N$ be the smallest positive integer such that $(b^N-1)\vartheta_0/K\in \Z$, where
the existence follows from the co-prime property between $K$ and $b$.
By Claim 4, there exists
$\omega_j\in \{-1,0, \ldots, b-2\},
1\le j\le N$, such that  the word $W:=\omega_1\omega_2\cdots\omega_{N}\ne 0$
and
\begin{equation} \label{App.thm.pf.part2eq8}
\frac{\vartheta_0}{K}=c.\omega_N\cdots\omega_2\omega_1\omega_N\cdots\omega_2\omega_1\cdots=
c+\frac{\sum_{j=1}^N \omega_j b^{j-1}}{b^N-1}
\end{equation}
for some integer $c\in \Z$. Let
$W^\prime=\omega_1^\prime \omega_2^\prime\cdots \omega_N^\prime$ be
so chosen that $\omega_j^\prime\in \{-1, 0, \ldots, b-2\}, 1\le j\le N$, and
\begin{equation} \label{App.thm.pf.part2eq9}
\sum_{j=1}^N(\omega_j'+\omega_j)b^{j-1}=\left\{\begin{array}
{ll} 0  & {\rm if}\  \sum_{j=1}^N \omega_j b^{j-1}\in
\frac{b^N-1}{b-1} [-1, 1)\\
b^N-1 & {\rm if}\  \sum_{j=1}^N \omega_j b^{j-1}\in
\frac{b^N-1}{b-1}
[1, b-2].
\end{array}\right.
\end{equation}
The existence of such a word $W^\prime$ follows from the observation that
$$\Big\{\sum_{j=1}^N \omega_j b^{j-1}, \omega_j\in \{-1, 0, \ldots, b-2\}\Big\}=
\Big(\frac{b^N-1}{b-1}[-1, b-2]\Big)\cap \Z.$$

Let $n> m/N$ and set $\zeta_{nN}=\zeta_00^{nN-m}\in \Sigma_q^{nN}$.
By Claim 3 and the $\kappa$-main branch assumption for $\zeta_0$,
there exists $\delta_{nN}\in \Sigma_q^{nN}$ such that
\begin{equation} \label{App.thm.pf.part2eq10}
K\Pi_{\tau, nN}(\delta_{nN})-\vartheta_0
\in b^{nN}\Z.
\end{equation}
Combining \eqref{App.thm.pf.part2eq8},
\eqref{App.thm.pf.part2eq9} and \eqref{App.thm.pf.part2eq10}
and recalling that $K$ and $b$ are coprime,  we obtain
\begin{equation*}\label{App.thm.pf.part2eq11}
(b^N-1)(\Pi_{\tau, nN}(\delta_{nN})-\tilde c)+
\sum_{j=1}^N \omega_j^\prime b^{j-1} \in b^{nN}\Z,
\end{equation*}
where
\begin{equation*}
\tilde c=\left\{\begin{array}
{ll} c & {\rm if}\  \sum_{j=1}^N \omega_j b^{j-1}\in
\frac{b^N-1}{b-1} [-1, 1)\\
c-1 & {\rm if}\  \sum_{j=1}^N \omega_j b^{j-1}\in
\frac{b^N-1}{b-1}
[1, b-2].
\end{array}\right.
\end{equation*}
Therefore
\begin{equation}\label {App.thm.pf.part2eq12}
\Pi_{\tau, nN}(\delta_{nN})-\tilde c-
\Big(\sum_{j=1}^N \omega_j^\prime b^{j-1}\Big) \big(1+b^N+\cdots+ b^{(n-1)N}\big)
\in b^{nN}\Z.
\end{equation}

By the construction of $\omega_j^\prime, 1\le j\le N$,
$\sum_{j=1}^N \omega_j^\prime b^{j-1}\in \frac{b^N-1}{b-1}(-1, b-2]$.
If either $\sum_{j=1}^N \omega_j^\prime b^{j-1}\in  \frac{b^N-1}{b-1}(-1, b-2)$
or $\sum_{j=1}^N \omega_j^\prime b^{j-1}=\frac{b^N-1}{b-1}(b-2)$ and $\tilde c\le 0$,
then for sufficiently large $k$,
$$\tilde c+\Big(\sum_{j=1}^N \omega_j^\prime b^{j-1} \Big)\big(1+b^N+\cdots+
 b^{(k-1)N}\big)=\sum_{j=1}^{kN} \theta_j b^{j-1}$$
 for some $\theta_j\in \{-1, 0, \ldots, b-2\}, 1\le j\le kN$,
 as it is contained in $[-(b^{kN}-1)/(b-1), (b^{kN}-1)(b-2)/(b-1)]$.
This together with \eqref{App.thm.pf.part2eq12} implies that
\begin{equation*}
\Pi_{\tau, nN}(\delta_{nN})=
\sum_{j=1}^{kN} \theta_j b^{j-1}+
\sum_{j=1}^N \omega_j^\prime b^{j-1} \big(b^{kN}+\cdots+ b^{(n-1)N}\big)
\end{equation*}
for $n\ge k$. Thus
there exists $\delta\in \Sigma_q^\infty$ such that
$\delta|_{nN}=\delta_{nN}$
and
\begin{equation*}
\tau(\delta|_{nN+j})=\omega_j^\prime, 1\le j \le N
\end{equation*}
for $n\ge k$, which proves the  desired conclusion.

Now consider the case that
 $\sum_{j=1}^N \omega_j^\prime b^{j-1}=\frac{b^N-1}{b-1}(b-2)$ and $\tilde c>0$.
 In this case,  $\omega_j^\prime=b-2$ for all $1\le j\le N$
 and  $N=1$ by the selection of the integer $N$.
 Further we obtain from  \eqref{App.thm.pf.part2eq12} that
 \begin{equation*}
 \Pi_{\tau, n}(\delta_{n})-\tilde c+1+\sum_{j=1}^{n}
 b^{j-1}\in b^n\Z,
 \end{equation*}
 which implies that
 there exists $\delta\in \Sigma_q^\infty$ such that
$\delta|_{n}=\delta_{n}$
and
$\tau(\delta|_{n})=-1$
for sufficiently large $n$, which proves the desired conclusion.

\subsection{Proof of Theorem \ref{counterexample}}
\label{example.section}
First we show that $\Lambda_{q, b}$ is a spectrum of the Cantor measure $\mu_{q, b}$.
Observe that  $\kappa$ is a maximal tree mapping,
every $\delta\in \Sigma_q^\ast$ is $\kappa$-regular, and
$\Lambda_{q, b}=\Lambda(\kappa)$. We then obtain from Theorem \ref{th1.6}
that
\begin{equation}\label{counterexample.pf.eq1}
 \Lambda_{q, b}\ \ \text{is a maximal orthogonal set of the Cantor measure}\
 \mu_{q, b}.\end{equation}
From  the definition of the maximal tree mapping $\kappa$ it follows that
\begin{equation}\label{counterexample.pf.eq2}
{\mathcal D}_{\kappa, \delta}\le {\mathcal D}_{\kappa, \delta}(0^\infty)\le K_\delta+1\le b-1 \ \ \text{for  all} \ \ \delta\in \Sigma_q^\ast,
\end{equation}
where $K_\delta$ is given in \eqref{counterexample.eq3}.
Therefore the spectral property for
$\Lambda_{q, b}$ holds by
 \eqref{counterexample.pf.eq1}, \eqref{counterexample.pf.eq2} and
  Theorem \ref{main.thm}.

Next we prove that $\Lambda_{q, b}/(b-1)$ is a  maximal orthogonal set for the Cantor measure $\mu_{q, b}$.
From  \eqref{orthogonalcharacterization} and the  spectral property for the  set $\Lambda_{q, b}$
We obtain   that
\begin{equation}
\label{counterexample.pf.eq5}
\Lambda_{q, b}-\Lambda_{q, b}\subset
\{ b^ja:\  0\le j\in \Z, a\in \Z\backslash q\Z\} \cup \{0\}.
\end{equation}
On the other hand,
$$
0\in \Lambda_{q, b}\subset \Z
$$
and for any $\delta\in \Sigma_q^*$,
\begin{equation} \label{counterexample.pf.eq6}
\Pi_{\kappa, \infty}(\delta)=\sum_{j=1}^\infty \kappa(\delta|_j)b^{j-1}\in \sum_{j=1}^\infty \kappa(\delta|_j) +(b-1)\Z=(b-1)\Z\end{equation}
by \eqref{counterexample.eq1}--\eqref{counterexample.eq3}.
Combining \eqref{counterexample.pf.eq5} and \eqref{counterexample.pf.eq6} leads to
\begin{equation*}
 (\Lambda_{q, b}-\Lambda_{q, b})/(b-1)\subset
\{b^ja:\   0\le j\in \Z, a\in \Z\backslash q\Z\} \cup \{0\},
\end{equation*}
and hence
$\Lambda_{q, b}/(b-1)$ is an orthogonal set for the Cantor measure $\mu_{q, b}$ by  \eqref{orthogonalcharacterization}.
Now we establish the maximality of the rescaled set $\Lambda_{q, b}/(b-1)$. Suppose, on the contrary, that there exists $\lambda_0\not\in \Lambda_{q, b}/(b-1)$ such that
 $\tilde \Lambda_{q, b}:=\Lambda_{q, b}/(b-1)\cup \{\lambda_0\}$
  is an orthogonal set for the Cantor measure
$\mu_{q, b}$. Then
\begin{eqnarray*} (b-1)\tilde \Lambda_{q, b}- (b-1)\tilde \Lambda_{q, b} & \subset &
(b-1)\big(\{ b^ja:\   0\le j\in \Z, a\in \Z\backslash q\Z\} \cup \{0\}\big)\\
&\subset &
\{ b^ja:\  0\le j\in \Z, a\in \Z\backslash q\Z\} \cup \{0\}
\end{eqnarray*}
and
$(b-1)\tilde \Lambda_{q, b}$ is an orthogonal set for
the Cantor measure $\mu_{q, b}$ by
\eqref{orthogonalcharacterization}.
This contradicts the spectral property for $\Lambda_{q, b}$.

Finally we prove that $\Lambda_{q, b}/(b-1)$ is not a spectrum of the Cantor measure $\mu_{q, b}$.
Let $\tau_{q, b}:\Sigma_q^\ast\to \{-1, 0, \ldots, b-2\}$
be the maximal tree mapping such that
$\Lambda_{q, b}/(b-1)=\Lambda(\tau_{q, b})$.
By Theorem \ref{main2.newthm}, the
non-spectral property for the set $\Lambda_{q, b}/(b-1)$
reduces to showing that
\begin{equation}\label{counterexample.pf.eq7}
 {\mathcal D}_{\tau_{q, b}, \delta}(0^\infty)\ge n
\end{equation}
for all $\delta \in \Sigma_q^n \backslash \Sigma_q^{n-1}, n\ge 2$, being $\tau_{q, b}$-regular.
Recall that
$\Lambda_{q, b}=\Lambda(\kappa)$.
This together with \eqref{counterexample.eq1} and \eqref{counterexample.eq2}
implies the existence of $\eta\in \Sigma_q^m, m\ge 1$,
such that
\begin{equation}  \label{counterexample.pf.eq8}
(b-1)\Pi_{\tau_{q, b}, \infty}(\delta)=\Pi_{\kappa, \infty}(\eta)
 =\sum_{j=1}^{m+b-2} d_j b^{j-1}+ q\cdot b^{2m+2b-2},
\end{equation}
where $d_j\in \{0, 1, \cdots, q\}$ for all $1\le j\le m+b-2$ and $d_m\in \{1, \ldots, q-1\}$.
 Write
 \begin{equation} \label{counterexample.pf.eq9}
\Pi_{\tau_{q, b}, \infty}(\delta)=\sum_{j=1}^\infty c_j  b^{j-1}
=\sum_{j=1}^M c_j  b^{j-1}
\end{equation}
where $c_j:=\tau_{q, b}(\delta|_j)\in \{-1, 0, \ldots, b-2\}$ and  $M\ge n$ is so chosen that
$c_M\ne 0$. The existence of such an integer follows from
$\tau_{q, b}(\delta|_n)\in \Z\backslash q\Z$ and
 $\tau_{q, b}(\delta|_j)=0$ for sufficiently large $j$.
Combining \eqref{counterexample.pf.eq8} and \eqref{counterexample.pf.eq9} leads to
 \begin{eqnarray}\nonumber
 \sum_{j=1}^M c_j  b^{j-1}
 &  =   &  \frac{1}{b-1} \Big(\sum_{j=1}^{m+b-2} d_j b^{j-1}+ q\cdot b^{m+b-2}\Big)+ q \sum_{j=m+b-2}^{2m+2b-3} b^j \\ \nonumber
 & \in & q \sum_{j=m+b-2}^{2m+2b-3} b^j+ \Big(0,\frac {b-2}{b-1}\Big)b^{m+b-2},
\end{eqnarray}
where the last inequality follows as $q\le b-3$.
%
This, together with $c_j\in \{-1, 0, 1,  \ldots, b-2\}, 1\le j\le M$,
 implies that
\begin{equation}\label{counterexample.pf.eq13} M=2m+2b-2\ \ {\rm and}\  \  c_j=q, m+b-2 < j\le M.
\end{equation}
On the other hand, for $\delta\in \Sigma_q^n\backslash \Sigma_q^{n-1}$
it follows from  the tree mapping property for $\tau_{q, b}$ that
$c_n\not\in q\Z$. Thus  $n\le m+b-2$ according to \eqref{counterexample.pf.eq13}.
Therefore
$${\mathcal D}_{\tau_{q, b}, \delta}(0^\infty)\ge M-(m+b-2)\ge n.$$
This proves \eqref{counterexample.pf.eq7} and then the conclusion that $\Lambda_{q, b}$
is not a spectrum of the Cantor set ${\mathcal \mu}_{q, b}$ by Theorem \ref{main2.newthm}.

%

%
%


\end{document}